\documentclass[11pt]{article}
\usepackage{amsmath,amssymb,amsthm,amsfonts}
\usepackage{mathtools}
\usepackage{bm}
\usepackage{cleveref}
\usepackage{natbib}
\usepackage[letterpaper, margin=1in]{geometry}

\newcommand{\F}{{\mathbb{F}}}
\newcommand{\Z}{{\mathbb{Z}}}
\newcommand{\E}{{\mathbb{E}}}

\DeclareMathOperator{\rank}{rank}

\DeclarePairedDelimiter\floor{\lfloor}{\rfloor}

\newtheorem{theorem}{Theorem}

\theoremstyle{definition}
\newtheorem{definition}{Definition}

\begin{document}
\title{A nearly tight upper bound on tri-colored sum-free sets in characteristic 2}
\author{Robert Kleinberg}
\maketitle

\begin{abstract}
A tri-colored sum-free set in an abelian group $H$ is a collection of ordered triples in $H^3$,
$\{(a_i,b_i,c_i)\}_{i=1}^m$, such that the equation $a_i+b_j+c_k=0$ holds if and only if
$i=j=k$. Using a variant of the lemma introduced by Croot, Lev, and Pach in their breakthrough
work on arithmetic-progression-free sets,
we prove that the 
size of any tri-colored sum-free set in $\F_2^n$ is bounded above by $6 \binom{n}{\floor*{n/3}}$.
This upper bound is tight, up to a factor subexponential in $n$: there exist tri-colored sum-free sets
in $\F_2^n$ of size greater than $\binom{n}{\floor*{n/3}} \cdot 2^{-\sqrt{16 n / 3}}$ 
for all sufficiently large $n$.
\end{abstract}

\section{Introduction} \label{sec:intro}
In a breakthrough paper, \citet*{croot:2016} applied the polynomial method to prove that for sufficiently 
large $n$, every set of more than $(3.62)^n$ elements of $(\Z/4\Z)^n$ contains a three-term
arithmetic progression. This was the first such bound of the form $c^n$ for a constant $c<4$. Soon
afterward, \citet{ellenberg:2016} and, independently, \citet{gijswijt:2016} extended the argument
to prove an upper bound of the form $c(p)^n$ on the size of any subset of $\F_p^n$ that is free of 
three-term arithmetic progressions, where $p$ is any odd prime and $c(p)$ is a constant strictly less
than $p$. \citeauthor{gijswijt:2016} provides the explicit bound $c(p) < e^{-1/18} p$.

In all of the aforementioned results, the upper bound obtained using the new methods is of the
form $C^n$ and the best known lower bound on the size of arithmetic-progression-free sets is of
the form $c^n$ for some $c < C$. Thus, in all known cases, there is still an exponential gap between
the best known upper and lower bounds for such sets.
In this note, we present a variant of the problem of finding
large sets that contain no three-term arithmetic progressions,
and we prove upper and lower bounds that differ by a sub-exponential
factor --- i.e., an upper bound of the form $c^{n+o(n)}$ and a lower
bound of the form $c^{n-o(n)}$, with the same constant $c$ appearing
as the base of the exponent in both bounds --- when the problem is restricted 
to the group $\F_2^n$. The upper bound proof is an application of the lemma
of \citet{croot:2016}, while the lower bound follows from a construction
due to \citet{fu:2014}, which in turn utilizes a construction from
\citet{coppersmith:1990}.

Since vector spaces over a field of characteristic 2 have no three-term arithmetic progressions,
it is not immediately clear how to generalize these questions to the case of characteristic 2.
The following generalization was proposed and analyzed
by \citet{blasiak:2016}.

\begin{definition} \label{def:tcsfs}
A tri-colored sum-free set in an abelian group $H$ is a collection $\{(a_i,b_i,c_i)\}_{i=1}^m$ of ordered triples in $H^3$
such that the equation $a_i+b_j+c_k=0$ holds if and only if
$i=j=k$. 
\end{definition}
Note that if $H$ is an abelian group of odd order and 
$A = \{a_1,\ldots,a_m\} \subseteq H$,  then $A$ contains no three-term arithmetic progressions
if and only if the set $\{(a_i,a_i,-2 a_i)\}$ is a tri-colored sum-free set. Thus, upper bounds on
the size of tri-colored sum-free sets immediately yield upper bounds on sets with no three-term
arithmetic progressions, but the definition of tri-colored sum-free sets is meaningful even when
$H = \F_2^n$.

\section{Upper Bound} \label{sec:ub}

To prove an upper bound on the size of tri-colored sum-free sets in $\F_p^n$, we will introduce
another closely related definition.

\begin{definition} \label{def:pms}
A {\em perfectly matched sequence} in an abelian group $H$ is a sequence
of ordered pairs $\{(a_i,b_i)\}_{i=1}^m$ in $H^2$ such that the equation 
$a_i+b_i=a_j+b_k$ has no solutions with $j \neq k$. The set $T = \{ a_i+b_i \mid
i=1,\ldots,m\}$ is called the {\em target set} of the perfectly matched sequence.
\end{definition}

Note that if $\{(a_i,b_i,c_i)\}$ is a tri-colored sum-free sequence of size $m$,
then $\{(a_i,b_i)\}$ is a perfectly matched sequence whose target set $T = \{-c_i\}$
has $m$ elements. The following theorem therefore yields an upper bound on
the size of tri-colored sum-free sequences.

\begin{theorem} \label{thm:pms}
Let $L_n$ denote the linear subspace of $\F_p[x_1,\ldots,x_n]$ 
spanned by monomials of the form
$\prod_{i=1}^n x_i^{\alpha_i}$, where $0 \leq \alpha_i < p$ for all $i$,
and let $L_{n,d}$ denote the subspace of $L_n$ spanned by monomials
of degree at most $d$. 
The target set of any perfectly matched sequence in $\F_p^n$
has at most $3 \dim L_{n,d}$ elements, where $d = \floor*{\frac13 (p-1) n}$.
\end{theorem}

\begin{proof}
The proof is a recapitulation of the proof of \citet{gijswijt:2016}, Theorem 2,
which corresponds to the special case when $a_i=b_i$ for all $i$. We reiterate 
the proof here to facilitate 
the task of verifying that \citeauthor{gijswijt:2016}'s proof extends to the
general case.

Let $V$ denote the vector space of polynomials $f \in L_{n,(p-1)n - d - 1}$
such that $f(x)=0$ for all $x \not\in T$. The dimension of 
$L=L_{n,(p-1)n-d-1}$ is equal to $p^n - \dim L_{n,d}$, and $V$ is obtained
from $L$ by imposing an additional $p^n-|T|$ linear constraints, one
for each $x \not\in T$. Hence $\dim V \geq |T| - \dim L_{n,d}$.

The evaluation map $V \to \F_p^T$ is injective --- see \citep{gijswijt:2016}, 
Proposition 1 --- hence there is a set $S \subseteq T$ of cardinality $|S| = \dim V$
such that the evaluation map $V \to \F_p^S$ is bijective. Choose a polynomial
$f \in V$ such that $f(x)=1$ for all $x \in S$, and consider the $(2n)$-variate
polynomial
\[
  g(x_1,\ldots,x_n,y_1,\ldots,y_n) = f(x+y).
\]
For a pair of multi-indices ${\alpha}, {\beta} \in \{0,\ldots,p-1\}^n$,
let $C_{{\alpha},{\beta}}$ denote the coefficient of the monomial
${x}^{{\alpha}} {y}^{{\beta}}$ in $g$. Our choice of
$d = \floor*{\frac13 (p-1) n}$ ensures that $(p-1)n - d - 1 \leq 2d+1$,
so $f \in L_{n,2d+1}$ and, consequently, for every monomial 
${x}^{{\alpha}} {y}^{{\beta}}$ occurring in $g$ either 
${x}^{{\alpha}}$ or ${y}^{{\beta}}$ has degree at most $d$.
Hence, the non-zero entries of $C$ belong to the union of a set of rows and
a set of columns each indexed by a set of $\dim L_{n,d}$ monomials. 
Accordingly, $\rank C \leq 2 \dim L_{n,d}$. On the other hand, the rank
of $C$ is bounded below by the rank of the matrix $M_{i,j} = f(a_i+b_j)$;
see \citet{gijswijt:2016}, Lemma 2.
By construction, $M_{i,j}=0$ when $i \neq j$ and $M_{i,j}=1$ when 
$i=j$ and $a_i+b_i \in S$. Hence,
\[
  |S| \leq \rank M \leq \rank C \leq 2 \dim L_{n,d}.
\]
Recalling that $|S| = \dim V \geq |T| - \dim L_{n,d}$, we obtain
the inequality $|T| \leq 3 \dim L_{n,d}$ as claimed.
\end{proof}

When $p=2$, we have $\dim L_{n,d} = \sum_{k=0}^{\floor*{n/3}} \binom{n}{k} <
2 \binom{n}{\floor*{n/3}}$. This bound, in conjunction with Theorem~\ref{thm:pms}, implies
the upper bound on tri-colored sum-free sets in $\F_2^n$ stated in the abstract.

\section{Lower Bound} \label{sec:lb}

Our lower bound on the size of tri-colored sum-free sets $\F_2^n$
recapitulates a construction due to \citet{fu:2014}
which, in turn, is based on a method
originating in the work of \citet{coppersmith:1990} on fast matrix multiplication.
We shall make use of the fact that the cyclic group $\Z / M\Z$, for large $M$,
has subsets of size $M^{1-o(1)}$ which contain no three-term arithmetic progressions.
The best known lower bound on the size of such subsets is the following
theorem of \citet{elkin:2011}; see also \citet{green:2010}. (In the theorem
statement, the expression $\log(\cdot)$ denotes the base-2 logarithm.)

\begin{theorem}[\citealp{elkin:2011}] \label{thm:elkin}
For all sufficiently large $M$, the group $\Z / M \Z$ has a subset of size
greater than $\log^{1/4}(M) \cdot 2^{-\sqrt{8 \log M}} \cdot M$ which
contains no three distinct elements in arithmetic progression.  
\end{theorem}

Assume for simplicity that $n$ is divisible by 3. (When $n$ is indivisible by 3,
we may take a large tri-colored sum-free set in $\F_2^{n'}$ for $n' = 3 \floor*{n/3}$
and ``pad'' each vector with 0's to obtain an equally large tri-colored sum-free set in 
$\F_2^n$.) Let $M$ be an odd integer greater than $4 \binom{2n/3}{n/3}$. 
Our tri-colored sum-free set will be constructed as a subset of the set $X$ 
of all triples $(a,b,c) \in \left( \{0,1\}^n \right)^3$ such that the vectors $a,b,c$
have Hamming weights $\frac{n}{3}, \frac{n}{3}, \frac{2n}{3}$, respectively, 
and $c=a+b$. Note that for any $(a,b,c) \in X$,
the equation $c = a+b$ holds regardless of whether the left and right sides 
are interpreted as vectors over $\F_2$ or over $\Z$. 

Letting $W = (\Z / M\Z)^{n+1}$ we now define three functions $h_0, h_1, h_2 :
\{0,1\}^n \times W \to \Z / M \Z$ as follows.
\begin{align*}
  h_0(a,w) = \sum_{s=1}^n a_s w_s, \qquad
  h_1(b,w) = \frac12 \left( w_0 + \sum_{s=1}^n b_s w_s \right), \qquad
  h_2(c,w) = w_0 + \sum_{s=1}^n c_s w_s.
\end{align*}
The function $h_1$ is well-defined because $\Z / M \Z$ is a cyclic group of odd order.
By construction, whenever $a,b,c$ are three vectors satisfying $a+b=c$
(over $\Z$), the values $h_0(a,w), h_1(b,w), h_2(c,w)$ are either identical or they form
an arithmetic progression in $\Z / M \Z$. Now, fix a set $B \subset \Z / M \Z$ that 
contains no three distinct elements in arithmetic progression.
For any $w \in W$ define sets $Y(w), Y_0(w), Y_1(w), Y_2(w), Y_3(w), Z(w)$ as follows.
\begin{align*}
  Y(w) & = \{ (a,b,c) \in X \mid h_0(a,w), h_1(b,w), h_2(c,w) \in B\} \\
  Y_0(w) &= \{ (a,b,c) \in Y(w) \mid \exists (b',c') \neq (b,c) \mbox{ s.t. } (a,b',c') \in Y(w) \} \\
  Y_1(w) &= \{ (a,b,c) \in Y(w) \mid \exists (a',c') \neq (a,c) \mbox{ s.t. } (a',b,c') \in Y(w) \} \\
  Y_2(w) &= \{ (a,b,c) \in Y(w) \mid \exists (a',b') \neq (a,b) \mbox{ s.t. } (a',b',c) \in Y(w) \} \\
  Z(w) &= Y(w) \setminus \left( Y_0(w) \cup Y_1(w) \cup Y_2(w) \right).
\end{align*}
We first claim that $Z(w)$ is a tri-colored sum-free set. The equation $a+b+c=0$ holds
in $\F_2^n$ for every $(a,b,c) \in Z(w)$, by construction, so we need only verify conversely that
for any three (not
necessarily distinct) elements 
$(a,b,c), (a',b',c'), (a'',b'',c'')$ of $Z(w)$, if the equation 
$a + b' + c'' = 0$ holds in $\F_2^n$ then all three of the given elements of $Z(w)$ 
are equal to one another. Indeed, our hypotheses about
$(a,b,c), (a',b',c'), (a'',b'',c'')$ imply all of the following conclusions
about $(a, b', c'')$:
\begin{enumerate}
\item $a$ and $b'$ have Hamming weight $n/3$, while $c''$ has Hamming weight $2n/3$;
\item $c'' = a  + b'$;
\item $h_0(a,w), h_1(b',w), h_2(c'',w) \in B$.
\end{enumerate}
In other words, $(a,b',c'')$ belongs to $Y(w)$. The fact that $(a,b,c) \not\in Y_0(w)$ now implies
that $(a,b,c) = (a,b',c'')$. Similarly, the facts that $(a',b',c') \not\in Y_1(w)$ and 
$(a'',b'',c'') \not\in Y_2(w)$ imply that $(a',b',c') = (a'',b'',c'') = (a,b',c'')$. Thus,
the three given elements of $Z(w)$ are all equal to one another, as required by the 
definition of a tri-colored sum-free set.

Let us now prove a lower bound on the expected cardinality of $Z(w)$ when $w$ is chosen
uniformly at random from $(\Z / M\Z)^{n+1}$. For a given element $(a,b,c) \in X$, the 
values $h_0(a,w), h_1(b,w), h_2(c,w)$
must either be equal to one another or they must form an arithmetic progression.
The set $B$ contains no three elements in arithmetic progression, so
the event that $h_0(a,w), h_1(b,w), h_2(c,w) \in B$ coincides with the
event that there exists $\beta \in B$ such that 
$h_0(a,w) = h_1(b,w) = h_2(c,w) = \beta$; furthermore, if any two of
 $h_0(a,w), h_1(b,w), h_2(c,w)$ are equal to $\beta$, then so is the third. 
For $w$ uniformly distributed in $(\Z / M\Z)^{n+1}$, the
values $h_0(a,w)$ and $h_2(c,w)$ are independent and uniformly
distributed in $\Z / M\Z$, so the probability of the event
$h_0(a,w)=h_2(c,w)=\beta$ is $M^{-2}$. Summing over 
all $\beta \in B$ and all $(a,b,c) \in X$, we find that the expected 
cardinality of $Y(w)$ is
\begin{equation} \label{lb.1}
  \E |Y(w)| = |X| \cdot |B| \cdot M^{-2} = \binom{n}{n/3} \cdot \binom{2n/3}{n/3} \cdot |B| \cdot M^{-2}.
\end{equation}
Similar reasoning allows us to derive an upper bound the expected 
cardinality of $Y_0(w)$. If $(a,b,c)$ belongs to $Y_0(w)$ it means that there
is some other element $(a,b',c') \in X$ and some $\beta \in B$ such that 
\begin{equation}  \label{lb.2}
  h_0(a,w) = h_1(b,w) = h_2(c,w) = h_1(b',w) = h_2(c',w) = \beta.
\end{equation} 
For $w$ uniformly distributed in $(\Z / M\Z)^{n+1}$, the
values $h_0(a,w), h_2(c,w),$ and $h_2(c',w)$ are independent and uniformly
distributed in $\Z / M\Z$; this is most easily verified by checking that
$h_0(a,w), h_2(c,w),$ and $h_2(c,w) - h_2(c',w)$ are independent and
uniformly distributed. Furthermore, if $h_0(a,w)=h_2(c,w)=h_2(c',w)=\beta$
then $h_1(b,w)=h_1(b',w)=\beta$, so the probability of the event indicated
in~\eqref{lb.1} is $|M|^{-3}$. Summing over all pairs of distinct elements
$(a,b,c),\, (a,b',c') \in X$ that share the same first coordinate, 
and all $\beta \in B$, we find that the expected cardinality of $Y_0(w)$ is at most 
\begin{equation} \label{lb.2}
  \E |Y_0(w)| \leq |X| \cdot \left( \binom{2n/3}{n/3} -1 \right) \cdot |B| \cdot M^{-3}
                     = \E |Y(w)| \cdot \frac{1}{M}  \left( \binom{2n/3}{n/3} -1 \right)
                     < \tfrac14 \E | Y(w) |
\end{equation}
where the last inequality is justified by our choice of $M > 4 \binom{2n/3}{n/3}$.
Analogous reasoning yields the bounds $\E |Y_1(w)|, \E |Y_2(w)| < \frac14 \E | Y(w) |$,
and hence
\[
  \E |Z(w)| \geq
  \E |Y(w)| - \E |Y_0(w)| - \E |Y_1(w)| - \E |Y_2(w)| > \tfrac14 \E |Y(w)| = 
  \frac14 \cdot \frac1M \binom{2n/3}{n/3} \cdot \frac{|B|}{M} \cdot \binom{n}{n/3}.
\]
If $n$ is sufficiently large, then for $M = 4 \binom{2n/3}{n/3} + 1$ and $B > \log^{1/4}(M) \cdot 2^{-\sqrt{8 \log M}} \cdot M$
we have
\[
  \frac14 \cdot \frac1M \binom{2n/3}{n/3} \cdot \frac{|B|}{M} > 2^{-\sqrt{16 n / 3}},
\]
hence
\[
  \E |Z(w)| > \binom{n}{n/3} \cdot 2^{-\sqrt{16 n / 3}} > \binom{n}{n/3}^{1-o(1)}
\]
as claimed.
%
%
%
%

\bibliographystyle{plainnat}
\bibliography{pmf-ub}

\end{document}